\documentclass[10pt]{article}
\usepackage{calc}
\usepackage{amsmath}
\usepackage{amssymb}
\usepackage{amsmath,amssymb,amsbsy,amsfonts,amsthm,latexsym,
amsopn,amstext,amsxtra,euscript,amscd}

\newtheorem{theorem}{Theorem}
\newtheorem{lemma}[theorem]{Lemma}
\newtheorem{corollary}[theorem]{Corollary}

\newcommand{\cA}{{\mathcal A}}
\newcommand{\cB}{{\mathcal B}}

\newcommand{\cF}{{\mathcal F}}
\newcommand{\cL}{{\mathcal L}}
\newcommand{\cO}{{\mathcal O}}
\newcommand{\cP}{{\mathcal P}}
\newcommand{\cQ}{{\mathcal Q}}
\newcommand{\cX}{{\mathcal X}}

\newcommand{\cJ}{{\mathcal J}}
\newcommand{\rF}{{\mathbb F}}

\newcommand{\ud}{{\,\mathrm d}}

\begin{document}
\title{On the distribution of sparse sequences in prime fields and applications}
\author{
{Victor C. Garc\'\i a}\\
\normalsize{Departamento de Ciencias B\'asicas}\\
\normalsize{Universidad
Aut\'onoma Metropolitana--Azcapotzalco}\\
\normalsize{C.P. 02200, M\'exico D.F., M\'exico}
\\
\normalsize{\tt vc.garci@gmail.com}}
\date{}

\pagenumbering{arabic}

\maketitle
\begin{abstract}
    In the present paper we investigate distributional properties of sparse sequences
    modulo almost all prime numbers. We obtain new results for a wide class of sparse sequences
    which in particular  find applications on additive problems and the discrete Littlewood
    problem related to lower bound estimates of the $L_1$-norm of trigonometric sums.
\end{abstract}


\section{Introduction}
   Throughout the paper $\{x_n\}$ is an increasing  sequence of positive integers.
   The study of distributional properties of  the sequence
   $$
      x_n \pmod p; \quad n=1,2,\ldots \,,
   $$
   and additive problems connected with such sequences are classical questions in number theory with a variety of results 
   in the literature. When $\{x_n\}$ grows rapidly the problem becomes harder for individual moduli, but it is possible to 
   obtain strong results modulo most of the primes $p.$ We mention the work of Banks, Conflitti, Friedlander and 
   Shparlinski~\cite{BanksConflittiFriedShpar}, where a series of results on distribution of Mersenne numbers 
   $M_q = 2^q-1$ in residue classes have been obtained. This question has also been considered by Bourgain in~\cite{Bourgain1}. 
   General results on distribution of sequences of type $2^{x_n},$ (and generally of the form $\lambda^{x_n}$),
   modulo most of the primes have 
   been obtained by Garaev and Shparlinski~\cite{GaraevShparlinski}, and by Garaev~\cite{Garaev1}. For instance,
    Garaev~\cite{Garaev1} has obtained a non-trivial bound  for the exponential sum
   $$
      \max_{(a,p)=1}\left|\sum_{n\le T}e^{2 \pi i \frac{a}{p}{\lambda}^{x_n}}\right|,
   $$
     for $\pi(N)(1+o(1))$ primes $p\le N$ and $T=N(\log N)^{2+\varepsilon},$ where
    $\{x_n\}$ is any strictly increasing sequence of positive integers satisfying $x_n\le n^{15/14 + o(1)}.$
    Banks, Garaev, Luca and Shparlinski~\cite{BanksGaraevLucaShparlinski} obtained uniform distributional properties
    of the sequences
   $$
     f_g(n)=\frac{g^{n-1}-1}{n}, \qquad h_g(n)=\frac{g^{n-1}-1}{P(n)},
   $$
     where $g$ and $n$ are positive integers, $n$ is composite and $P(n)$ is the largest prime factor of $n.$

\bigskip
  
   Now consider a simpler sequence 
   $$
     2^n \pmod p ; \qquad n=1, 2,\ldots \, .
   $$
     From a result of Erd\"os and  Murty~\cite{ErdosMurty} it is well-known that $2$ has the multiplicative
     order $t_p\ge N^{1/2 + o(1)}$
     for $\pi(N)(1+o(1))$ primes $p\le N$. Combining this  with a
      result of Glibichuk~\cite{Glibichuk} it follows  that for almost all primes $p$ every
     residue class modulo $p$ can be represented in the form
   $$
     2^{n_1} + \ldots +2^{n_8} \pmod p,
   $$
     for certain positive integers $n_1,\ldots, n_8.$ Garc\'\i a, Luca and Mej\'\i a~\cite{GarciaLucaMejia}
     have applied similar arguments to obtain analogous results for
     the sequence of Fibonacci numbers
   $$
     F_n \pmod p; \qquad n=1,2,\ldots \,,
   $$
     where
  $$
     F_{n+2}=F_{n+1}+F_n, \qquad n\ge 1,
  $$
     with $F_1=F_2=1.$ They proved that for almost all primes $p,$ every residue class modulo $p$
     is a  sum of 32 Fibonacci numbers.

     In the present paper using a different approach we obtain new results on  additive properties
     for general sparse sequences for almost all the prime moduli.  In particular
     we prove that for $\pi(N)(1+o(1))$ primes $p \le N$ every residue class is a sum of 16
     Fibonacci numbers $F_n,$ with $n \le N^{1/2 + o(1)},$ improving upon the mentioned result of
     Garc\'\i a, Luca and Mej\'\i a. Moreover, we
     establish that for any $\varepsilon >0$  there is an integer $s\le 100/\varepsilon$ such that for
     $\pi(N)(1+o(1))$ primes, $p\le N,$ every residue class can be written
     as
   $$
     F_1 +\ldots +F_{s}\pmod p,
   $$
     with  $1\le n\le N^{\varepsilon}.$  We note that
     the value $s$ has the optimal order $s=\cO(1/\varepsilon).$

\bigskip
     
    From the work of Karatsuba~\cite{Karatsuba} it is known the connection between additive problems and the Littlewood 
    problem on lower bound estimates for the $L_1$-norm of
    exponential sums.  Namely, for any coefficents $\gamma_n,$ $|\gamma_n|=1,$ and any strictly increasing sequence of
    integers $\{f(n)\},$ Karatsuba established 
  \begin{equation} \label{ineq:KaratsubaL1}
     \int \limits_0^1\left|  \sum_{n=1}^{N}\gamma_n e^{2\pi i \alpha f(n)}\right| \ud \alpha \ge
      \left(N^3/J\right)^{1/2},
  \end{equation}
    where $J$ denotes the number of solutions of the diophantine equation
  $$
    f(n)+f(m)=f(k)+f(\ell); \qquad 1 \le n,m,k,\ell \le N.
  $$  
    Solving the Littlewood conjecture, Konyagin~\cite{Konyagin}, and  McGehee, Pigno and
    Smith~\cite{McGePigSmith} proved that
 \begin{equation}\label{eq:LittlewoodConj}
   \int \limits_{0}^{1}\left| \sum_{n=1}^{N}e^{2\pi i \alpha f(n)}\right| \ud \alpha \gg \log N,
 \end{equation}
    where $f(n)$ is an integer valued function. This bound reflects the best possible lower bound
     in general settings, as it shown by the example $f(n)=n$.
   However, due to the connection with certain additive problems, for a very wide class of integer valued sequences $f(n),$
   estimate~\eqref{eq:LittlewoodConj}
   has been improved, see, for example, Garaev~\cite{Garaev}, Karatsuba~\cite{Karatsuba}
   and Konyagin~\cite{Konyagin1}.

   Green and Konyagin~\cite{GreenKony} established a variant of the Littlewood problem in prime fields $\rF_p.$
   One of their results claims that if $\cA$ is a subset of $\rF_p,$ with $|\cA|=(p-1)/2,$ then
 $$
   \frac{1}{p}\sum_{x=0}^{p-1}\left|\sum_{a \in \cA}e^{2 \pi i x \frac{x}{p}a}\right|
   \gg (\log p)^{1/3 - \varepsilon}.
 $$
 
   One can use a version of Karatsuba's inequality~\eqref{ineq:KaratsubaL1}
    to get a variety of result for specific sequences. For instance, we employ a recent result of Bourgain 
    and Garaev~\cite{BourgainGaraev} on additive energy of 
   the set $g^x \pmod p; \, 1\le x\le N,$ and show that for any
   $N < p^{1/2}$ we have the bound
  $$
   \frac{1}{p}\sum_{x=0}^{p-1}\left|\sum_{n \le N}e^{2 \pi i  \frac{x}{p}g^n}\right| \gg N^{1/48 +o(1)}.
  $$
    For the sequence $\{F_n\}$  of Fibonacci numbers we shall prove the following result: given any positive real
    $\gamma < 1/3$ there are positive constants $c_1=c_1(\gamma), c_2=c_2(\gamma)$ such that for 
    $\pi(N)(1+o(1))$ primes  $p\le N$ the following estimate holds
  $$
     c_1 N^{\gamma/2}\le \frac{1}{p}\sum_{x=0}^{p-1}\left|\sum_{n\le N^{\gamma}}
     e^{2 \pi i \frac{x}{p}F_n}\right|
      \le c_2 N^{\gamma/2}.
  $$

     {\bf Acknowledgement.} The author is grateful to  M.~Z.~Garaev for sharing his thoughts which has led to the  present improvement of the result of~\cite{GarciaLucaMejia}.

 \section{Formulation of results}

   Throughout the paper $N$ and $M$ always denote positive large parameters.
   Let $\cX$ be any subset of $\{1,\ldots, 10^M\}.$
   The first result of our present paper relies on ideas of arithmetic combinatorics.

   \begin{theorem}\label{thm:average}
      Let $\cJ(N)$ be the number of solutions
      of
  \begin{equation}\label{eq:sim}
    x\equiv y \pmod p; \qquad x,y \in \cX, \quad p\le N.
  \end{equation}
    The following asymptotic formula holds
   \begin{equation}\label{asymp:average}
        \cJ(N)=\pi(N)|\cX| + \cO\left(\frac{|\cX|^2M}{\log M}\right).
   \end{equation}
\end{theorem}

   Using Theorem~\ref{thm:average} we can get the following result on the value set 
   of any sequence modulo most of the primes $p$.

\begin{theorem}\label{thm:main}
     For $\pi(N)(1+\cO(1/\Delta))$
     prime numbers $p\le N,$ we have the following asymptotic formula
   \begin{equation}\label{eq:Size}
      \#\{x \pmod p \;:\;x \in \cX\;\}= |\cX| +
         \cO \left(\frac{|\cX|}{1 + \frac{\pi(N)\log M}{M|X|\Delta}}\right),
   \end{equation}
     where $\Delta=\Delta(N)$ is any function with $\Delta\to \infty.$
\end{theorem}
   In particular we have the following corollary which can be applied for a large class of sparse sequences.
\begin{corollary}\label{coro:asympSparseSet}
     If $M |\cX|\Delta^2 \le \pi(N)\log M ,$ then
   \begin{equation}\label{eq:SizeSparse}
      \#\{x \pmod p \;:\;x \in \cX\;\}= |\cX|\left( 1 +
         \cO \left({\Delta}^{-1}\right)\right ).
   \end{equation}
\end{corollary}

\subsection{Additive properties of Fibonacci numbers}

    Theorem~\ref{thm:main} finds application on additive problems for well known very fast increasing sequences.
    For example the following theorems on additive properities of the Fibonacci sequence $\{F_n\}.$
\begin{theorem}\label{thm:Waring_0}
    For \mbox{$\pi(N)(1+o(1))$}
    primes $p\le N,$ every integer $\lambda$ can be written as
  $$
    F_{n_1}+\ldots +F_{n_{16}}\equiv \lambda \pmod p,
  $$
    where $1 \le n_1, \ldots, n_{16} \le N^{1/2+o(1)}.$
\end{theorem}
   This improves the result of Garc\'\i a, Luca and Mej\'\i a~\cite{GarciaLucaMejia} on the representation of
   any residue class $\lambda$ in the  form
 $$
    F_{n_1}+\ldots +F_{n_{32}}\equiv \lambda \pmod p,
 $$
    for certain integers $n_1, \ldots, n_{32},$ for almost all primes $p.$
    Moreover, combining Theorem~\ref{thm:main} with exponential sum techniques we
    obtain a more general result.

 \begin{theorem}\label{thm:Waring_1}
    Let $0 < \varepsilon <1/2.$ There is an integer $s < 100/\varepsilon$  such that for
    $\pi(N)(1 + o(1))$ primes $p \le N,$ every integer $\lambda$ can be written as
  $$
    F_{n_1}+\ldots +F_{n_s}\equiv \lambda \pmod p,
  $$
    where $n_i \le N^{\varepsilon},$ $i=1,\ldots, s.$
 \end{theorem}

    Observe that the number of terms on the sumatory has the expected order, apart from the value 100.
    Indeed we obtain $s=4([8/\varepsilon]-1).$
    However, we do not consider a reduction to be
    essential in this paper.

\subsection{Application to the discrete Littlewood \mbox{problem}}

The following theorem presents an
   application of a result of Bourgain and Garaev~\cite[Theorem 1.4]{BourgainGaraev}.

\begin{theorem}\label{thm:L1PrimitiveRoot}
   Let $g$ be a primitive root modulo $p.$ If $N < p^{1/2}$ then
 $$
   \frac{1}{p}\sum_{x=0}^{p-1}\left|\sum_{n\le N}e^{2 \pi i  \frac{x}{p}g^n}\right|\gg N^{1/48 + o(1)}.
 $$
\end{theorem}

  Regarding the Fibonacci sequence, we prove the following theorem.

\begin{theorem}\label{thm:L1Fibonacci}
     Let $0<\gamma <1/3.$ There are two positive absolute constants $c_1=c_1(\gamma), c_2=c_2(\gamma)$
     such that for  $\pi(N)(1 + o(1))$ primes, $p \le N,$ we have
   $$
     c_1 N^{\gamma/2}\le \frac{1}{p}\sum_{x=0}^{p-1}\left|\sum_{n\le N^{\gamma}}
     e^{2 \pi i \frac{x}{p}F_n}\right|
      \le c_2 N^{\gamma/2}.
   $$
\end{theorem}

\section{Notation and lemmas}

   For given subsets $\cA$ and $\cB$ of the residue field $\rF_p$ and any
   integer $k\ge 2,$ as usual, we denote
 \begin{align*}
     \cA+\cB &= \{a + b \;:\;a\in \cA,\; b \in \cB\},\\
     \cA\cdot \cB &= \{a b \;:\;a\in \cA,\; b \in \cB\},\\
     k \cA &= \{a_1+\ldots+a_k\;:\; a_1,\ldots, a_k \in \cA\}.
 \end{align*}
    For any finite subset of integers $\cX$ we denote
 $$
   \cX \pmod p = \{x \pmod p \;:\; x \in \cX\}.
 $$
    The next lemma is a result of Glibichuk~\cite{Glibichuk}.
\begin{lemma}\label{lemma:Glibichuk}
     Let $\cA, \cB$ be subsets of $\rF_p$ such that $|\cA||\cB|> 2p.$ Then
  $$
     8 \cA\cdot \cB = \rF_p. \qed
  $$
\end{lemma}

   Given a fixed prime number $p,$ we denote by $t_p$
   the {\it multiplicative order} of 2 modulo $p.$ That is
 $$
   t_p = \min \{\ell \;:\;2^{\ell}\equiv 1 \pmod p\}.
 $$
   As we mentioned in the introduction, the result of  Erd\"os--Murty~\cite{ErdosMurty} establish
   that for $\pi(N)(1+o(1))$ primes, $p\le N,$  we have $t_p > N^{1/2}e^{(\log N)^{\rho_0}},$
   with some sufficiently small $\rho_0 >0.$  We present an analogous result  for the
   {\it order of appearance,} defined by
 $$
   z(k)=\min\{\ell \;:\; F_{\ell}\equiv 0 \pmod k\},
 $$
   where $k$ is a fixed integer $k\ge 2$ and
   $F_n$ denotes the $n$th term of the sequence of Fibonacci numbers.

\begin{lemma}\label{lemma:OrderApp}
    For almost all primes
    $p \le N,$ we have
  $$
     z(p)\ge  N^{1/2}e^{(\log N)^{\rho}},
  $$
    with some appropriate $\rho >0.$
\end{lemma}

   We require the following lemma which follows from exponential sums estimates, see for
   example the proof of~\cite[Theorem1.1]{GaraevKueh} or~\cite{Sarkozy}.
\begin{lemma}\label{lemma:ternary}
    Let $X, Y$ and $Z$ be subsets of $\{0,1,\ldots, p-1\}.$
    Denote by $T$ the number of solutions of the congruence
  \begin{equation}\label{eq:trigsum}
    x y + z_1 + z_2 \equiv \lambda \pmod p,
  \end{equation}
    where
  $$
    x \in X, \quad y \in Y, \quad z_1,z_2 \in Z.
  $$
    Then, the asymptotic formula
  $$
    T = \frac{|X||Y||Z|^2}{p} + \theta\sqrt{p|X||Y|}|Z|, \quad |\theta|\le 1,
  $$
    holds uniformly over $\lambda.$ In particular Eq.~\eqref{eq:trigsum} has
    solution if $|X||Y||Z|^{2}>  p^{3}.$
\end{lemma}
\bigskip

     We shall use some results concernig the values of Fibonacci sequence.
   \begin{align}
     F_{u+v} &=\frac{1}{2}(F_uL_{v} + L_u F_{v}), \label{eq:FibIdentity1}\\
     F_{u-v}&=\frac{(-1)^{v}}{2}(F_uL_{v} - L_u F_{v}),\label{eq:FibIdentity2}
   \end{align}
     where $\{L_m\}$ is the Lucas sequence given by
   $$
      L_{m+2}=L_{m+1}+L_m, \quad L_1=1, \,L_2=3.
   $$

   The following lemma is due to Bourgain and Garaev~\cite{BourgainGaraev}

\begin{lemma}\label{lemma:EnergyPrimRoots}
   Let $g$ be a fixed primitive root modulo $p.$ Let $1 < M < p^{1/2}$ and denote by $T$ be the number of solutions of the congruence
 $$
   g^x +g^y \equiv g^z + g^w \pmod p; \qquad 1 \le x,y,z,w \le M.
 $$
   Then
 $$
  T < M^{3 - 1/24 + o(1)}.
 $$
\end{lemma}

\section{Proof of Theorems}

\subsection{Proof of Theorem~\ref{thm:average}}
   If $x=y$ then Eq.~\eqref{eq:sim} has $\pi(N)|\cX|$ solutions. Therefore
 \begin{equation}\label{eq:J}
   \cJ(N)=\pi(N)|\cX| + \cJ',
 \end{equation}
   where $\cJ'$ denotes the number of solutions of~\eqref{eq:sim} subject to $x\neq y.$ Given
   $x,y$ in $\cX$ with $x\neq y,$ the equation
 $$
   pk = x-y, \qquad p \le N,
 $$
   has at most $\omega(|x-y|)$ solutions, where $\omega(n)$ denotes the number of prime divisors of $n.$
   If \mbox{$4\le |x-y|\le 10^M$ }, using the well-known estimate $\omega(n)\ll (\log n)/(\log \log n),$ 
   we obtain that~\eqref{eq:sim} has at most $\cO(|\cX|^2M/\log M)$ solutions. Otherwise, if $0< |x-y|<4,$ then~\eqref{eq:sim}
   has no more than $\cO(|\cX|)$ solutions. Thus
 $$
   \cJ' \ll |\cX|^2\frac{M}{\log M}.
 $$
   Inserting this upper bound for $\cJ'$ in~\eqref{eq:J}, Theorem~\ref{thm:average} follows. $\qed$

\subsection{Proof of Theorem~\ref{thm:main}}

   Before the proof, we shall introduce the following lemma

\begin{lemma}\label{lemma:asympJ_p}
        Let $J_p$ be the number of solutions of the congruence
   \begin{equation}\label{eq:J_p}
      x \equiv y \pmod p; \quad x,y \in \cX.
   \end{equation}
     For $\pi(n)=(1+\cO(1/\Delta))$ primes $p\le N$ we have
  \begin{equation}\label{asymp:J_p}
     J_p = |\cX|+\cO \left(\frac{|\cX|^2M}{\pi(N)\log M}\Delta\right).
  \end{equation}
\end{lemma}
\begin{proof}
     Note that $J_p \ge |\cX|,$ because the case $x=y$ satisfies~\eqref{eq:J_p}. It is clear that
  $$
    \cJ(N)= \sum_{p \le N} J_p.
  $$
    Denote by $\cP$ the set of prime numbers $p\le N$ such that
  $$
    J_p - |\cX| > \frac{|\cX|^2 M}{\pi(N)\log M}\Delta.
  $$
   If $p$ runs through the set $\cP,$ recalling that $J_p - |\cX|\ge 0,$ we get
  $$
    |\cP| \frac{|\cX|^2 M}{\pi(N)\log M}\Delta \le \sum_{p \in \cP}(J_p -|\cX|)\le
    \sum_{p \le N}(J_p -|\cX|) = \cJ(N) - \pi(N)|\cX|.
  $$
    Thus,  applying Theorem~\ref{thm:average}, we derive that
   $$
     |\cP|\ll \frac{\pi(N)}{\Delta}.
   $$
     Therefore, if $\cQ$ denotes the number of primes $p\le N$ such that
   $$
     J_p - |\cX| \le \frac{|\cX|^2 M}{\pi(N)\log M}\Delta,
   $$
     then
   $$
     |\cQ|= \pi(N)-|\cP|= \pi(N)(1+\cO(\Delta^{-1})).
   $$
\end{proof}
    Theorem~\ref{thm:main} follows from the relation
   $$
     \#\{x \pmod p \;:\; x \in \cX\} \ge \frac{|\cX|^2}{J_p}. \qed
   $$

\subsection{Proof of Theorem~\ref{thm:Waring_0}}
    Lemma~\ref{lemma:OrderApp} allow us to establish the order of the value set of  the  Fibonacci sequence
    for most primes
   \begin{equation*}
      \#\{F_n \pmod p \;:\; n \le \delta N^{1/2 }\} \asymp \delta N^{1/2 },
   \end{equation*}
     where $\delta=\delta(N)=e^{(\log N)^{\rho}}$ and $\rho>0 $  is the refered constant in Lemma~\ref{lemma:OrderApp}.
     In order to establish the last relation, it is sufficent to prove that for
   \begin{equation*}
       \cF=\{F_{2n}  \;:\; \delta N^{1/2} /10< n \le \delta N^{1/2} /5\},
   \end{equation*}
      we have
   \begin{equation}\label{eq:ValueSetFib}
     |\cF \pmod p| = |\cF| = \frac{\delta N^{1/2}}{10} + \cO(1).
   \end{equation}
     Let $n,n'$ be positive integers such that
   \begin{equation}\label{eq:I}
     F_{2n}\equiv F_{2n'}\pmod p; \quad  \delta N^{1/2} /10<n,n' \le \delta N^{1/2} /5.
   \end{equation}
      Without loss of generality we can assume that $n\ge n'.$ Substituting $u= n +n'$ and $v=n-n'$
      in~\eqref{eq:FibIdentity1} and~\eqref{eq:FibIdentity2},  we can obtain
   {\small
   $$
      F_{2n} - F_{2 n'} =
      \frac{1}{2}\left((1-(-1)^{n-n'})F_{n+n'}L_{n-n'}+ (1+(-1)^{n-n'})L_{n+n'}F_{n-n'}\right).
   $$ }
      Suppose that $n- n' \equiv 0 \pmod 2,$ then from Eq.~\eqref{eq:I} follows
   $$
     p | L_{n+n'}F_{n-n'}.
   $$
      If $n \neq n',$ then  $0 < n-n'< N^{1/2}\delta \le z(p),$ which implies
      $(p,F_{n-n'})=1.$
       Thus
   $$
     p| L_{n+n'}, \quad\textrm{in particular }\,\; p|F_{n +n'}L_{n+n'},
   $$
     where $F_{n +n'}L_{n+n'}=F_{2(n+n')}.$ Hence $p|F_{2(n +n')},$
     with $2(n+ n') < z(p).$ This contradicts the choice of $z(p).$ Therefore in the case
     $n - n' \equiv 0 \pmod 2$
     Eq.~\eqref{eq:I} has only trivial solutions $n=n'.$  Similarly,  it is possible to verify that 
     \eqref{eq:I} has not solutions if 
     $n - n' \equiv 1 \pmod 2.$

\bigskip

     Now, consider the subset of Lucas sequence
   $$
      \cL = \{L_{2m} \;:\; 1\le m\le N^{1/2}/\sqrt{\delta}\}.
   $$
     Taking in Theorem~\ref{thm:main}; $M=N^{1/2}/\sqrt{\delta}$ and $\Delta={\delta}^{1/4}$ we obtain
    \begin{equation}\label{eq:ValueSetLucas}
        | \cL \pmod p| =\frac{N^{1/2}}{\sqrt{\delta}}(1+\cO({\delta}^{-1/4})).
    \end{equation}
     Observe that equalities~\eqref{eq:ValueSetFib} and~\eqref{eq:ValueSetLucas} are valid respectively
     for most primes. Thus,  for $\pi(N)(1 + o(1))$ primes $p \le N$ we have
   $$
     |\cF \pmod p||\cL \pmod p|\gg \sqrt{ \delta} N \ge 2p.
   $$
     Applying Lemma~\ref{lemma:Glibichuk},
     we obtain that for almost all primes $p$  every integer $\lambda$ can be written as
   $$
     F_{2n_1}L_{2m_1} +\ldots+F_{2n_8}L_{2m_8}\equiv \lambda \pmod p,
   $$
     where
   $$
     N^{1/2}\delta /10< n_i  \le N^{1/2}\delta /5, \quad 1 \le m_i \le N^{1/2}/\sqrt{\delta}, \quad 1 \le i \le 8.
   $$
     Using the identity
   $$
     F_uL_{v}= F_{u+v} + (-1)^{v}F_{u-v},
   $$
     for every $1\le  i\le  8$ we get
   $$
     F_{2n_i}L_{2m_i}=  F_{2(n_i+m_i)} + F_{2(n_i-m_i)}.
   $$
     Thus, Theorem~\ref{thm:Waring_0} follows. $\qed$
\subsection{Proof of Theorem~\ref{thm:Waring_1}}
     Let $k$ be the minimal integer such that ${1}/{(k+2)} < {\varepsilon}/{8}.$  Define the sets
   \begin{align*}
     X &=\{F_{2n_1-1}+ \ldots + F_{2n_k-1}\;:\;1\le n_1, \ldots, n_k \le N^{\frac{1}{k+2}}\}, \\
     Y &= \{L_m \;:\;\frac{1}{2}N^{\frac{7}{k+2}} < m \le N^{\frac{7}{k+2}}\}, \\
     Z &=\{F_{2\ell_1}+ \ldots + F_{2\ell_k}\;:\; 1 \le \ell_1, \ldots, \ell_k \le N^{\frac{1}{k+2}} \}.
   \end{align*}
     Observe that $|Y| \gg N^{\frac{7}{k+2}}$ and exist a positive constant  $c=c(k) < 1$ such that
   $$
     |X|, |Z| \ge cN^{\frac{k}{k+2}}.
   $$
     In order to estimate the value set of $X \pmod p$ note that if $x\in X,$ then $x \le 10^{\log k N^{1/(k+2)}}.$ Thus,
     applying Corollary~\ref{coro:asympSparseSet}  with $M=\log k N^{1/(k+2)},$ $\cX=Z $ and $\Delta = (\log N)^A,$
     (for any integer $A>0$), we have that for most of primes $p\le N$
   $$
     |X \pmod p | =|X|(1 + o(1)).
   $$
     Analogously, we can obtain
   $$
     |Y \pmod p|=|Y|(1 + o(1)), \qquad |Z \pmod p|=|Z|(1 + o(1)),
   $$
     for almost all primes respectively. Therefore, there is a constant $c_1=c_1(k),$ $0 < c_1 <1,$ such that
     for $\pi(N)(1+o(1))$  primes $p\le N$ we have
   $$
     |X \pmod p||Y \pmod p||Z \pmod p|^2 \ge c_1 N^{3 + \frac{1}{k+2}}> p^{3 + \frac{1}{k+2}}.
   $$
     Applying Lemma~\ref{lemma:ternary} it follows that for almost all primes every integer $\lambda$ can be
     represented as
   \begin{equation}\label{eq:TernaryFibonacci}
      \sum_{i=1}^{k}L_mF_{2n_i-1}+ \sum_{j=1}^{k}(F_{2\ell_j}+F_{2\ell'_j}) \equiv \lambda \pmod p,
   \end{equation}
     where
   $$
     \frac{1}{2}N^{\frac{7}{k+2}} < m \le N^{\frac{7}{k+2}}, \quad
     1\le n_i\le N^{\frac{1}{k+2}}, \quad
     1 \le \ell_j,\ell'_j \le N^{\frac{1}{k+2}}, \quad (1 \le i,j \le k).
   $$
     We recall the identity
   $$
      L_u F_v = F_{u+v}+(-1)^{v+1}F_{u-v}.
   $$
     Thus, for every $1 \le i \le k$ in~\eqref{eq:TernaryFibonacci} we get
   $$
     L_mF_{2n_i-1}=F_{m + 2n_i - 1}+ F_{m - 2n_i +1}.
   $$
     taking $s=4k$ (that  is, $s = 4([ 8/\varepsilon]-1)$), we conclude that for almost all primes
     every residue class $\lambda$ has a representation
     in the form
   $$
     F_{n_1}+\ldots +F_{n_s}\equiv \lambda \pmod p,
   $$
     for some integers
   $$
     1\le n_1, \ldots, n_s \le N^{\varepsilon}. \qed
   $$
\subsection{Proof of Theorem~\ref{thm:L1PrimitiveRoot}}
    Note that the congruence
  $$
    g^x \equiv g^y \pmod p; \qquad1 \le x,y \le N,
  $$
    has exactly  $N$ solutions. Therefore,
  $$
    N = \frac{1}{p}\sum_{x=0}^{p-1}\left|\sum_{n\le N}e^{2 \pi i \frac{x}{p}g^n}\right|^2 =
    \frac{1}{p}\sum_{x=0}^{p-1}\left|\sum_{n\le N}e^{2 \pi i \frac{x}{p}g^n}\right|^{2/3}
    \left|\sum_{n\le N}e^{2 \pi i \frac{x}{p}g^n}\right|^{4/3},
  $$
    using H\"older's inequality we obtain
  \begin{eqnarray*}
     N &\le  &\frac{1}{p}\left(\sum_{x=0}^{p-1}\left|\sum_{n\le N}e^{2 \pi i \frac{x}{p}g^n}\right|\right)^{2/3}
     \left(\sum_{x=0}^{p-1}\left|\sum_{n\le N}e^{2 \pi i \frac{x}{p}g^n}\right|^4\right)^{1/3}\\
       & \le &{T^{1/3}}\left(\frac{1}{p}\sum_{x=0}^{p-1}\left|\sum_{n\le N}e^{2 \pi i \frac{x}{p}g^n}\right|\right)^{2/3},
  \end{eqnarray*}
    where, $T$ denotes the number of solutions of the congruence
  $$
     g^x + g^y \equiv g^z + g^w \pmod p; \qquad  1\le x,y,z,w \le N.
  $$
    Thus, from Lemma~\ref{lemma:EnergyPrimRoots} we know that $T < N^{3 - 1/24 + o(1)}.$ Therefore
  $$
    \frac{1}{p}\sum_{x=0}^{p-1}\left|\sum_{n\le N} e^{2 \pi i }\right| > N^{1/48 - o(1)}. \qed
  $$

\subsection{Proof of Theorem~\ref{thm:L1Fibonacci}}

    Observe that the congruence
  $$
    F_n\equiv F_{n'}\pmod p; \qquad 1 \le n,n' \le N^{\gamma},
  $$
    has at least $N^{\gamma}$ solutions. Therefore
  $$
    N^{\gamma}\le \frac{1}{p}\sum_{x=0}^{p-1}
    \left|\sum_{n\le N^{\gamma}}e^{2 \pi i \frac{x}{p}F_n}\right|^2.
  $$
    From H\"older's inequality, as in the proof of Theorem~\ref{thm:L1PrimitiveRoot}, it is possible to
    obtain
  \begin{equation}\label{ineq:L1Fibonacci}
     N^{\gamma} \le  T_p^{1/3}\left(\frac{1}{p}\sum_{x=0}^{p-1}
     \left|\sum_{n\le N^{\gamma}}e^{2 \pi i \frac{x}{p}F_n}\right|\right)^{2/3},
  \end{equation}
    where $T_p$ denotes the number of solutions of the congruence
  $$
    F_{n_1}+F_{n_2}\equiv F_{m_1}+F_{m_2}\pmod p; \quad 1 \le n_1,n_2,m_1,m_2 \le N^{\gamma}.
  $$
    Let
  $$
    \cX = \{F_{n_1}+F_{n_2} \;:\; 1 \le n_1,n_2 \le N^{\gamma}\}.
  $$
    Then $|\cX|\asymp N^{2\gamma}.$
    Applying  Lemma~\ref{lemma:asympJ_p} with $M=N^{\gamma}$ and $\Delta=N^{(1-3\gamma)/2}$
    we get, for $\pi(N)(1+o(1))$ primes $p\le N,$ the estimation
  $$
    T_p \le N^{2\gamma}\left(1 + N^{-(1-3\gamma)/2}\right).
  $$
    Combining this estimation with relation~\eqref{ineq:L1Fibonacci} we conclude that there is a
    positive constant $c_1(\gamma)$ such that
  $$
    \frac{1}{p}\sum_{x=0}^{p-1}
     \left|\sum_{n\le N^{\gamma}}e^{2 \pi i \frac{x}{p}F_n}\right| \ge c_1(\gamma)N^{\gamma/2}.
  $$
    Finally, to obtain an upper bound of the same order, using the Cauchy-Schwartz inequality we have
  \begin{equation}\label{ineq:UpperBoundFibonacci}
    \left( \frac{1}{p}\sum_{x=0}^{p-1}
     \left|\sum_{n\le N^{\gamma}}e^{2 \pi i \frac{x}{p}F_n}\right|\right)^2 \le
     \frac{1}{p}\sum_{x=0}^{p-1}
     \left|\sum_{n\le N^{\gamma}}e^{2 \pi i \frac{x}{p}F_n}\right|^2,
  \end{equation}
    where the right term is, indeed, the number of solutions of the congruence
  $$
     F_n\equiv F_m \pmod p; \qquad 1 \le n,m \le N^{\gamma}.
  $$
    Applying again Lemma~\ref{lemma:asympJ_p} with $M=N^{\gamma}$ and $\Delta=N^{(1-2\gamma)/2},$
    we obtain, for $\pi(N)(1+o(1))$ primes $p\le N,$ the estimation
  $$
    \frac{1}{p}\sum_{x=0}^{p-1}
     \left|\sum_{n\le N^{\gamma}}e^{2 \pi i \frac{x}{p}F_n}\right|^2 \le N^{\gamma}\left(1 + N^{-(1-2\gamma)/2}\right)
     \le c_2(\gamma) N^{\gamma},
  $$
    for some positive constant $c_2(\gamma).$ Putting together with~\eqref{ineq:UpperBoundFibonacci} and
    taking square root we conclude the proof. $\qed$


\begin{thebibliography}{999}

\bibitem{BanksConflittiFriedShpar}  W.~D.~Banks, A.~Conflitti, J.~B.~Friedlander and I.~E.~Shparlinski,
 `Exponential sums over Mersenne numbers,'  {\it Compos. Math.,} {\bf 140} (1), 15--30  (2004).

\bibitem{BanksGaraevLucaShparlinski}W.~D.~Banks, M.~Z.~Garaev, F.~Luca and I.~E.~Shparlinski,
`Uniform distribution of
fractional parts related to pseudprimes,' {\it Canad. J. Math.,} {\bf 61} (3), 481--502 (2009).

\bibitem{Bourgain1} J.~Bourgain, `Estimates on exponential sums related to the Diffie--Hellman distributions,'
{\it Geom. Funct. Anal.,} {\bf 15} (1), 1--34 (2005).

\bibitem{BourgainGaraev} J.~Bourgain, M.~Z.~Garaev, `On a variant of sum-product estimates and explicit exponential
sum bounds in prime fields,'  {\it Math. Proc. Cambridge Philos. Soc.,}  {\bf 146} (1), 1--21 (2009).

 \bibitem{ErdosMurty} P.~Erd\H os and M.~R.~ Murty, `On the order of $a\pmod p$', in {\it Number theory
 (Ottawa, ON, 1996)\/}, 87--97, CRM Proc. Lecture Notes {\bf 19}, Amer. Math. Soc., Providence, RI, 1999.

\bibitem{Garaev} M.~Z.~Garaev, `Upper bounds for the number of solutions of a diophantine equation,'  {\it Trans. Amer.
Math. Soc.,} {\bf 357}, 2527--2534  (2005).

\bibitem{Garaev1}M.~Z.~Garaev, `The large sieve inequality for the exponential sequence $\lambda^{[O(n^{15/14+o(1)})]}$ modulo primes'  {\it Canad. J. Math.,} {\bf 61} (2), 336--350 (2009).

\bibitem{GaraevKueh} M.~Z.~Garaev and  Ka--Lam Kueh, `Distribution of special sequences modulo a large prime'
{\it Int. J. Math. Math. Sci.,} {\bf 50}, 3189-3194 (2003).

\bibitem{GaraevShparlinski}M.~Z.~Garaev and I.~E.~Shparlinski, `The large sieve inequality with exponential functions
and the distribution of Mersenne numbers modulo primes,' {\it Int. Math. Res. Notices,} (39), 2391--2408 (2005).

\bibitem{GarciaLucaMejia} V.~C.~Garcia, F.~Luca and V.~J.~Mejia, `On sums of Fibonacci numbers modulo $p$,'
{\it Bull. Aust. Math. Soc.}, (to appear).

\bibitem{Glibichuk} A.~A.~Glibichuk, `Combinatorial properties of sets of residues modulo a prime an the Erd\H os--Graham
problem', {\it Mat. Zametki \/} {\bf 79}:3, 384--395 (2006); English transl., {\it Math. Notes} {\bf 79}:3--4, 356--365 (2006).

\bibitem{GreenKony} B.~Green and S.~V.~Konyagin, `On the Littlewood problem modulo a prime,'  Canad. J. Math.
{\bf 61} (1), 141--164 (2009).

\bibitem{Karatsuba} A.~A.~Karatsuba, `An estimate of the $L_1$-norm of an expontential sum,' Math. Notes, {\bf 64},
401--404 (1998).

\bibitem{Konyagin} S.~V.~Konyagin, `On the problem of Littlewood,' Izv. Acad. Nauk SSSR Ser. Mat. [Math. USSR-Izv.],
{\bf 45 }(2), 243--265 (1981).

\bibitem{Konyagin1}S.~V.~Konyagin, `An estimate of the $L_1$-norm of an exponential sum,' The Theory of Approximations
of Functions and Operators. Abstracts of Papers of the International Conference
Dedicated to StechkinÕs 80th Anniversary [in Russian]. Ekaterinburg (2000), pp. 88-89.

\bibitem{McGePigSmith} O.~C.~McGehee, L.~Pigno and B.~Smith, `Hardy's inequality and the $L_1$ norm of exponential
sums,' Ann. of Math. (2), {\bf 113}(3), 613--618 (1981).

\bibitem{Sarkozy} A.~S\'ark\"ozy, `On sums and products of residues modulo $p$,'  {\it Acta Arith.,}  {\bf 118} (4),
403--409 (2005).

\end{thebibliography}
\end{document}